\documentclass[11pt]{article}  

\usepackage{amsmath, amsthm, amsfonts, amssymb}
\usepackage{hyperref}
\usepackage{indentfirst} 
\usepackage{marvosym}

\usepackage[a4paper]{geometry}

\newtheorem{theoremletters}{Theorem}

\newtheorem{corollaryletters}[theoremletters]{Corollary}

\newtheorem{lemma}{Lemma}[section]
\newtheorem{theorem}[lemma]{Theorem}
\newtheorem{proposition}[lemma]{Proposition}
\newtheorem{corollary}[lemma]{Corollary}
\renewenvironment{proof}[1][\proofname]{{\sc #1. }}{\qed}
\theoremstyle{definition}
\newtheorem{example}[lemma]{Example}
\newtheorem{remark}[lemma]{Remark}

\newcommand{\abs}[1]{\ensuremath{\left| #1 \right|}}
\newcommand{\op}{\operatorname}
\newcommand{\ce}[2]{\operatorname{C}_{#1}(#2)}

\newcommand{\ze}[1]{\operatorname{Z}(#1)}
\newcommand{\fra}[1]{\op{\Phi}(#1)}
\newcommand{\fit}[1]{\operatorname{F}(#1)}

\begin{document}

\title{\bf Square-free class sizes in products of groups}

\author{\sc M. J. Felipe $\cdot$ A. Mart\'inez-Pastor $\cdot$ V. M. Ortiz-Sotomayor
\thanks{The first author is supported by Proyecto Prometeo II/2015/011, Generalitat Valenciana (Spain), and the second author is supported by Proyecto MTM2014-54707-C3-1-P, Ministerio de Econom\'ia, Industria y Competitividad (Spain). The results in this paper are part of the third author's Ph.D. thesis, and he acknowledges the predoctoral grant ACIF/2016/170, Generalitat Valenciana (Spain). The authors wish to thank John Cossey for helpful conversations during his last visit to Valencia. \newline
\rule{6cm}{0.1mm}\newline
Instituto Universitario de Matem\'atica Pura y Aplicada (IUMPA), Universitat Polit\`ecnica de Val\`encia, Camino de Vera, s/n, 46022, Valencia, Spain. \newline
\Letter: \texttt{mfelipe@mat.upv.es}, \texttt{anamarti@mat.upv.es}, \texttt{vicorso@doctor.upv.es}
}}

\date{}

\maketitle

\begin{abstract}
\noindent We obtain some structural properties of a factorised group $G=AB$, given that the conjugacy class sizes of certain elements in $A \cup B$ are not divisible by $p^2$, for some prime $p$. The case when $G=AB$ is a mutually permutable product is especially considered. \\


\smallskip

\end{abstract}


\section{Introduction} 

All groups considered thorogouth this paper are finite. Over the last years, many authors have investigated the influence of conjugacy class sizes on the structure of finite groups. In the meantime, numerous studies in the framework of group theory have focused in factorised groups. In this setting, a central question is how the structure of the factors affects the structure of the whole group, in particular when they are connected by certain permutability properties. The purpose of this paper is to show new achievements which combine both current perspectives in finite groups. More precisely, our aim is to get some information about a factorised group, provided that the conjugacy class sizes of some elements of its factors are square-free.

The earlier starting point of our investigation can be traced back to the paper of Chillag and Herzog (\cite{CH}), where the structure of a group in which all elements have square-free conjugacy class sizes was first analysed. Next, in \cite{CW}, Cossey and Wang localised one of the main theorems in \cite{CH} for a fixed prime $p$, that is, they considered conjugacy class sizes not divisible by $p^2$, for certain prime $p$. Later on, this study was improved by  Li in \cite{LI}, and by Liu, Wang, and Wei in \cite{LWW}, by replacing conditions on all conjugacy classes by those referring only to conjugacy classes of either $p$-regular elements or prime power order elements, using the classification theorem of finite simple groups (CFSG). These authors also first obtained some preliminary results in factorised groups.  This research was extended in 2012 by  Ballester-Bolinches, Cossey and Li in \cite{BCL}, through mutually permutable products. More recently, in 2014, Qian and Wang (\cite{QW}) have gone a step further by considering just conjugacy class sizes of $p$-regular elements of prime power order (although not in factorised groups).

In the context of factorised groups, and aiming to obtain criteria for products of supersoluble subgroups to be supersoluble, several authors have considered products in which certain subgroups of the factors permute  (see \cite{BEA} for a detailed account). In this scene, we are interested in \emph{mutually permutable products}, factorised groups $G = AB$ such that the subgroups $A$ and $B$ are mutually permutable, i.e.,  $A $ permutes with every subgroup of $B$ and $B$ permutes with every subgroup of $A$ (see also \cite{BH}). Obviously, if $A$ and $B$ are normal in $G$, then they are mutually permutable.

We recall that, for a group $G$, the set $x^G = \lbrace g^{-1}xg \: : \: g\in G\rbrace$ is  the \emph{conjugacy class} of the element $x \in G$, and $\abs{x^{G}}$ denotes the \emph{conjugacy class size} of $x$. If $p$ is a prime number, we say that  $x \in G$ is a $p$-\emph{regular} element if its order is not divisible by $p$, and that it is a $p$\emph{-element} if its order is a power of $p$. Moreover,  if $n$ is an integer,  let $n_p$ denote the highest power of $p$ dividing $n$. The $m$th group of order $n$ in the \verb+SmallGroups+ library \cite{GAP} of GAP will be identified by $n\#m$. The remainder notation is standard and is taken mainly from \cite{DH}. We also refer to this book for details about classes of groups.

In this paper, motivated by the above development, at first we focus on the case of $p$-groups,  extending for factorised groups the well-known Knoche's theorem (see \cite{KNO}).

\begin{theoremletters}
Let $p$ be a prime number and let $P=AB$ be a $p$-group such that $p^2$ does not divide $\abs{x^P}$ for all $x\in A\cup B$. Then $P' \leqslant \fra{P} \leqslant \ze{P}$, $P'$ is elementary abelian and $\abs{P'}\leq p^2$.
\label{knoche}
\end{theoremletters}

Our next goal in the paper is to prove the following theorem, regarding mutually permutable products.

\begin{theoremletters}
\label{teoA}
Let $G = AB$ be the mutually permutable product of the subgroups $A$ and $B$, and let $p$ be a prime such that $\operatorname{gcd}(p - 1, \abs{G}) = 1$. If $p^{2}$ does not divide $\abs{x^{G}}$ for any $p$-regular element $x\in A\cup B$ of prime power order, then:

\emph{(1)} $G$ is soluble.
	
\emph{(2)} $G$ is $p$-nilpotent.
	
\emph{(3)} The Sylow $p$-subgroups of $G/\op{O}_{p}(G)$ are elementary abelian.
\end{theoremletters}

In the particular case when $G=A=B$, we recover \cite[Theorem A]{QW} (see Section \ref{sec1}, Corollary \ref{CorTeoAQW}). We remark that both results use the CFSG. 

Moreover, in relation to the third assertion, we have found a gap in one of the statements in \cite[Theorem 1]{CW}, as it is reported in Remark \ref{remark1} (a).

On the other hand, we point out that it is possible to find examples of groups factorised as a product of two (mutually permutable) subgroups which satisfy the hypotheses of Theorem \ref{teoA} for some fixed prime $p$ and, however, there exist elements $x \in A \cup B$ such that $p^2$ divides either $\abs{x^A}$ or $\abs{x^B}$ (see Remark \ref{remark1} (b)).

The next theorem generalises the last assertion of \cite[Theorem 1.3]{BCL} regarding $p$-soluble groups, by considering only prime power order elements:

\begin{theoremletters}
\label{teoB}
Let $G = AB$ be the mutually permutable product of the subgroups $A$ and $B$, and let $p$ be a prime. Suppose that for every prime power order $p$-regular element $x\in A\cup B$, $\abs{x^{G}}$ is not divisible by $p^{2}$. If $G$ is $p$-soluble, then $G$ is $p$-supersoluble.
\end{theoremletters}

In the line of \cite[Theorem 1]{CH} and \cite[Theorem 2]{CW}, if we consider all prime numbers, then we obtain some information about the structure of the derived subgroup of a factorised group $G$.

\begin{theoremletters}
\label{teoD}
Let $G = AB$ be the product of the subgroups $A$ and $B$, and assume that $G$ is supersoluble. Suppose that every prime power order element $x\in A\cup B$ has square-free conjugacy class size. Then:

\emph{(1)} $G'$ is abelian.
	
\emph{(2)} The Sylow subgroups of $G'$ are elementary abelian.

\emph{(3)} $\fit{G}'$ has Sylow $p$-subgroups of order at most $p^2$, for every prime $p$.
\end{theoremletters}

If we limit our conditions only to $p$-regular elements, as a consequence of Theorems \ref{teoA} and \ref{teoB}, we obtain the following result which extends \cite[Corollary 1.5]{BCL} (see Corollary \ref{cor1.5BCL}) for prime power order elements, and also a theorem of \cite{LI}, for products of groups.

\begin{theoremletters}
\label{teoC}
Let $G = AB$ be the mutually permutable product of the subgroups $A$ and $B$. Suppose that for every prime $p$ and every prime power order $p$-regular element $x\in A\cup B$, $\abs{x^{G}}$ is not divisible by $p^{2}$. Then $G$ is supersoluble, and $G/\op{F}(G)$ has elementary abelian Sylow subgroups.
\end{theoremletters}

We remark that the first statement in Theorem \ref{teoD} is not further true under the weaker hypotheses of the above theorem, even for arbitrary groups not necessarily factorised, as pointed out in \cite{QW}.  Indeed, as a result of Theorem \ref{teoC}, the supersolubility condition in Theorem \ref{teoD} can be exchanged by the mutual permutability of the factors.

On the other hand, with the stronger assumption that all $p$-regular elements of the factors (not only those of prime power order) have conjugacy class sizes not divisible by $p^2$, we get extra information about the orders of the Sylow $p$-subgroups of $G/F(G)$, extending partially \cite[Theorem 2]{CW}.

\begin{theoremletters}
\label{teoE}
Let $G = AB$ be the mutually permutable product of the subgroups $A$ and $B$. Suppose that for every prime $p$ and every $p$-regular element $x\in A\cup B$, $\abs{x^{G}}$ is not divisible by $p^{2}$. Then the order of a Sylow $p$-subgroup of $G/\op{F}(G)$ is at most $p^2$.
\end{theoremletters}

In summary, when dealing with mutually permutable products, the next corollary follows directly from the above theorems.

\begin{corollaryletters}
Let $G = AB$ be the mutually permutable product of the subgroups $A$ and $B$. Suppose that $\abs{x^{G}}$ is square-free for each element $x\in A\cup B$. Then $G$ is supersoluble, and both $G/\op{F}(G)$ and $G'$ have elementary abelian Sylow subgroups. Moreover, $G'$ is abelian, and both groups $G/\op{F}(G)$ and $\fit{G}'$ have Sylow $p$-subgroups of order  at most $p^2$, for each prime $p$. 
\end{corollaryletters}

In Section \ref{sec1} we prove Theorems \ref{knoche}, \ref{teoA} and \ref{teoB}, which refer to class sizes not divisible by $p^2$, for a fixed prime $p$. Theorems \ref{teoC}, \ref{teoD} and \ref{teoE}, which consider square-free conjugacy class sizes (for all primes), are proved in Section \ref{sec2}.  In both cases we will illustrate the scope of the results presented with some examples.


\section{Preliminary results}

We use the following elementary properties frequently, sometimes without further reference.

\begin{lemma}
Let $N$ be a normal subgroup of a group $G$, and let $p$ be a prime. Then:

\emph{(a)} $\abs{x^N}$ divides $\abs{x^G}$, for any $x\in N$.
	
\emph{(b)} $\abs{(xN)^{G/N}}$ divides $\abs{x^G}$, for any $x\in G$.
	
\emph{(c)}  If $xN$ is a $p$-element of $G/N$, then there exists a $p$-element $x_{1}\in G$ such that $xN = x_{1}N$.
\end{lemma}

We need specifically the following fact about Hall subgroups of factorised groups. It is a reformulation of \cite[1.3.2]{AMB} which is convenient for our purposes. We recall that a group is  a D$_{\pi}$-group, for a set of primes $\pi$, if every $\pi$-subgroup is contained in a Hall $\pi$-subgroup, and any two Hall $\pi$-subgroups are conjugate. In particular, all $\pi$-separable groups are  D$_{\pi}$-groups for any set of primes $\pi$, and all groups are D$_{\pi}$-groups when $\pi$ consists of a single prime.

\begin{lemma}\label{1.3.2}
Let $G=AB$ be the product of the subgroups $A$ and $B$. Asume that $A, B$, and $G$ are D$_{\pi}$-groups for a set  of primes $\pi$. Then there exists a Hall $\pi$-subgroup $H$ of $G$ such that $H= (H \cap A)(H \cap B)$, with $H \cap A$ a Hall $\pi$-subgroup of $A$ and $H \cap B$ a Hall $\pi$-subgroup of $B$.
\end{lemma}

We collect here some results on mutually permutable products, which will be very useful along the paper.

\begin{lemma}
\label{mutperm}
Let the group $G = AB$ be the product of the mutually permutable subgroups $A$ and $B$. Then we have:

\emph{(a)} \emph{(\cite[4.1.10]{BEA})} $G/N$ is the product of the mutually permutable subgroups $AN/N$ and $BN/N$.
	
\emph{(b)} \emph{(\cite[4.1.21]{BEA})} If $U$ is a subgroup of $G$, then $(U\cap A)(U\cap B)$ is a subgroup, and $U\cap A$ and $U\cap B$ are mutually permutable. Moreover, if $N$ is a normal subgroup of $G$, then $(N\cap A)(N\cap B)$ is also normal in $G$. 
\end{lemma}

\begin{theorem}\emph{(\cite[Theorem 1]{BH})}
\label{prodnotrivial}
Let the non-trivial group $G=AB$ be the product of the mutually permutable subgroups $A$ and $B$. Then $A_{G}B_{G}$ is not trivial.
\end{theorem}

The following lemma will be essential in the proofs of our theorems.

\begin{lemma} \emph{(\cite[Lemma 2.4]{BCL})}
\label{lemabcl}
Let $p$ be a prime, and $Q$ be a $p'$-group acting faithfully on an elementary abelian $p$-group $N$ with $\abs{[x, N]} = p$, for all $1\neq x\in Q$. Then $Q$ is cyclic.
\end{lemma}

The next result is the first assertion of Theorem A in \cite{QW}, which uses the CFSG.

\begin{theorem}
\label{lemaqwa}
Let $G$ be a group. For a fixed prime $p$ with $\operatorname{gcd}(p - 1, \abs{G}) = 1$, if $p^{2}$ does not divide $\abs{x^{G}}$ for any $p$-regular element $x\in G$ of prime power order, then $G$ is soluble.
\end{theorem}

Finally, the later lemma, which is a nice result due to Isaacs, will be very useful in the proof of Theorem \ref{teoD}. 

\begin{lemma}\emph{(\cite[4.17]{ISA})}
\label{isaac_lem}
Let $K$ be an abelian normal subgroup of a finite group $G$, and let $x\in G$ be non-central. Then $\abs{\ce{G}{x}}<\abs{\ce{G}{y}}$, where $y=[k, x]$ and $k\in K$ is arbitrary.
\end{lemma}


\section{Class sizes not divisible by \texorpdfstring{$p^2$}{p2}, for a fixed prime  \texorpdfstring{$p$}{p}}
\label{sec1}

The well-known Knoche's theorem (see \cite{KNO}) asserts that if $P$ is a $p$-group, $p$ a prime, then the conjugacy class sizes of $P$ are square-free if, and only if, $\abs{P'}\leq p$. We begin this section by proving Theorem \ref{knoche}, which clearly extends it for factorised groups.

\medskip

\begin{proof}[Proof of Theorem \ref{knoche}]
Since $\abs{P:\ce{P}{x}}\leq p$ for each $x\in A\cup B$, it follows $\fra{P}\leqslant\ce{P}{x}$. Therefore, $\fra{P}$ commutes with both $A$ and $B$, so $P' \leqslant \fra{P} \leqslant \ze{P}$. Hence $P/\ze{P}$ is elementary abelian, and $x^{p} \in \ze{P}$ for all $x\in P$. Since $P' \leqslant \ze{P}$ and $[x, y]^p = [x^p, y] = 1$ (see \cite[A - 7.3(a)]{DH}) for any $x, y \in P$, it follows that $P'$ is elementary abelian. Now it remains to prove that $\abs{P'}\leq p^2$.

Let $[x, y]$ be a generator of $P'$. Since $P' \leqslant \ze{P}$ and $y = y_ay_b$ with $y_a \in A$ and $y_b \in B$, then $[x, y] = [x, y_ay_b] = [x, y_b][x, y_a]^{y_b} = [x, y_b][x, y_a] \in [P, B][P, A]$ . Thus $P' = [P, B][P, A]$. Clearly, $[P, B]$ is elementary abelian. Suppose $[P, B]\neq1$, and let $1 \neq [x, z]$ and $1 \neq [x', z']$ be two generators of $[P, B]$, with $x, x' \in P$ and $z, z' \in B$. We distinguish three cases in order to prove that $\langle [x, z]\rangle = \langle [x', z']\rangle$:

i) Suppose first $z, z' \in B\setminus \ze{B}$. Let $b\in B\setminus(\ce{P}{z}\cup\ce{P}{z'})$. Since $\abs{P:\ce{P}{z}} = p$ then $P = \ce{P}{z}\langle x\rangle$. Moreover, $b\notin\ce{P}{z}$ implies that $1\neq [b, z]=[x^{i}t, z] = [x^i, z] = [x, z]^i$, where $b = tx^i$ with $t\in \ce{P}{z}$. On the other hand, $P = \ce{P}{b}\langle z\rangle$ so $z' = z^jk$, with $k\in\ce{P}{b}$. Hence $1\neq[b, z'] = [b, z^jk] = [b, z]^j = [x, z]^{i+j}$. Finally, as $b \in P = \ce{P}{z'}\langle x'\rangle$, then $b = (x')^m s$ with $s\in \ce{P}{z'}$. Therefore $1 \neq [x, z]^{i+j} = [b, z'] = [(x')^m s, z'] = [x', z']^m$, and recall that $[x,z]$ and $[x',z']$ both have order $p$. Thus, $\langle [x, z] \rangle= \langle [x', z'] \rangle$. 

ii) Now suppose $z, z' \in \ze{B}$. Then $B\leqslant \ce{P}{z} \cap \ce{P}{z'}$. There exists $w \in P\setminus (\ce{P}{z}\cup \ce{P}{z'})$. Therefore, $w = w_aw_b$ with $w_a\in A$ and $w_b\in B \leqslant \ce{P}{z} \cap \ce{P}{z'}$, so $w_a\in A\setminus(\ce{P}{z} \cup \ce{P}{z'})$. Arguing analogously as in case i) with $w_a$ instead of $b$, we conclude that $\langle [x, z] \rangle= \langle [x', z'] \rangle$ too.

iii) Finally, suppose $z\in B\setminus \ze{B}$ but $z' \in \ze{B}$. Let $z'' = zz' \in B\setminus\ze{B}$. Therefore, we have $[x', z''] = [x', z'][x', z]$. If $[x', z'']=1$ then $1 \neq [x', z']^{-1} = [x', z]$, and applying  case i) to both $[x, z]$ and $[x', z]$ we conclude that they generate the same cyclic group of order $p$. On the other hand, if $[x', z'']\neq 1$ and $[x', z] =1$, then we apply again the first case. Finally, if both $[x', z''] \neq 1 \neq [x', z]$ then they generate the same cyclic group by case i) again. Thus $1\neq \langle [x', z']\rangle = \langle [x', z''][x', z]^{-1}\rangle \leqslant \langle [x', z'']\rangle$. Since the last one has order $p$, it follows $\langle [x', z']\rangle = \langle [x', z'']\rangle$. So we have $\langle [x', z']\rangle= \langle [x', z'']\rangle$, which is equal to $\langle [x, z]\rangle$ by i) again. 

In conclusion, if $[P,B] \neq 1$, then it has order $p$. Analogously with $[P, A]$. Hence $\abs{P'} =\abs{[P, B][P, A]} \leq p^2$ and this establishes the result.
\end{proof}

\begin{example}
The converse of the above result is not true in general, in contrast to Knoche's theorem. Let $P$  be the group of the Small groups library of \verb+GAP+ with identification number $32\#35$, which is the product of a cyclic group of order $4$ and a quaternion group of order $8$. Then its derived group is $P' = C_2 \times C_2$, and $P' = \fra{P} = \ze{P}$. Nevertheless, there are elements in the quaternion group with conjugacy class size equal to 4.
\end{example}

\begin{example}
Let $G= Q_8 \times D_8$ be the direct product of a quaternion group and a dihedral group of order $8$. Then every element contained in each factor has conjugacy class size equal to either $1$ or $2$, so Theorem \ref{knoche} applies. However, there are elements in $G$ with conjugacy class size divisible by $4$, and Knoche's result cannot be applied.
\end{example}

Now we proceed with a key result in the sequel.

\begin{proposition}
\label{proptec}
Let $G$ be a group, and let $p$ be a prime. Suppose that $N$ is an abelian minimal normal subgroup of $G$, which is a $p'$-group. Then:

\emph{(1)} If $G$ is $p$-nilpotent, and $\abs{x^{G}}$ is not divisible by $p^2$ for every element $x\in N$, then $\abs{\op{O}_{p}(G/N\op{O}_{p}(G))}\leq p$.
	
\emph{(2)} If $K/N\op{O}_{p}(G) = \op{O}_{p}(G/N\op{O}_{p}(G))$ has order $p$, and $P$ is a Sylow $p$-subgroup of $K$, then $\ce{N}{P}=1$.
\end{proposition}

\begin{proof}
(1) Suppose that the result is not true, and let $G$ be a counterexample of minimal order. Since the hypotheses are inherited by quotients, we may assume by standard arguments that $\op{O}_{p}(G)=1$, and then also $\fra{G}=1$. Since $N$ is abelian, by a Gasch\"utz's result (\cite[4.4]{HUP}) $N$ is complemented, that is, $G = NS$ with $N\cap S = 1$. We may assume that $G/N$ is not a $p'$-group, so $\op{O}_{p}(G/N) \cong \op{O}_{p}(S) \neq 1$ by the minimality of $G$. Let $P$ be a Sylow $p$-subgroup of $S$ (so $P$ is a Sylow $p$-subgroup of $G$). Hence it follows $\op{O}_{p}(S) \cap \operatorname{Z}(P) \neq 1$. Let $Z$ be a minimal normal subgroup of $\op{O}_{p}(S) \cap \operatorname{Z}(P)$. Since $S$ is $p$-nilpotent, we get $S=PL$, where $L$ is normal in $S$ and $P\cap L =1$. It follows that 
$[L, Z] \leqslant [L, \op{O}_{p}(S)]\leqslant L\cap \op{O}_{p}(S) = 1,$
 so $Z\leqslant \operatorname{Z}(S)$. Note that $\operatorname{C}_{N}(Z)$ is normal in $G = SN$. Consequently, by the minimality of $N$, we have either $\operatorname{C}_{N}(Z)=1$ or $\operatorname{C}_{N}(Z)=N$. If $\operatorname{C}_{N}(Z) = N$, then $Z\leqslant\operatorname{Z}(G)$, which implies that $Z\leqslant\op{O}_{p}(G)=1$, a contradiction. So we may affirm $\operatorname{C}_{N}(Z) = 1$, for every minimal normal subgroup $Z$ of $\op{O}_{p}(S) \cap \operatorname{Z}(P)$.

Now let $1\neq x \in N$ such that a Sylow $p$-subgroup of $\operatorname{C}_{G}(x)$, say $P_0$, is contained in $P$, so $P_0=\operatorname{C}_{P}(x)$. By the hypotheses, $\abs{x^{G}}_p = \abs{G:\operatorname{C}_{G}(x)}_p = \abs{P:P_0}$ is not divisible by $p^2$, so it follows either $\abs{P:P_0} = 1$ or $\abs{P:P_0} = p$. The first case yields $P = \operatorname{C}_{P}(x)$ and then $x\in \operatorname{C}_{N}(Z) =1$, a contradiction. Therefore, we may assume that $\abs{P:P_0} = p$, and so $P_0$ is normal in $P$. In addition, since $\op{O}_{p}(S) \cap \operatorname{Z}(P)$ is abelian, by the minimality of $Z$, we have either $P_0\cap Z = 1$ or $P_0\cap Z = Z$. The last case gives $Z\leqslant \operatorname{C}_{P}(x)$, a contradiction again. Hence, $P_0\cap Z = 1$ and it follows that $P = P_0\times Z$ and $\abs{Z}=\abs{P:P_0}=p$. We only need to see that $Z=\op{O}_{p}(S)$ to finish the proof. 

Note that $\operatorname{Z}(P) = \operatorname{Z}(P)\cap P_0Z = Z(\operatorname{Z}(P)\cap P_0),$ so it follows  $\op{Z}(P)\cap \op{O}_{p}(S) = Z(\operatorname{Z}(P)\cap P_0)\cap \op{O}_{p}(S) = Z(\operatorname{Z}(P) \cap P_0 \cap \op{O}_{p}(S)).$
If $\operatorname{Z}(P) \cap P_0 \cap \op{O}_{p}(S) \neq 1$, since it is normal in $\op{O}_{p}(S) \cap \operatorname{Z}(P)$, we can choose a minimal normal subgroup $Z_{1}$ of $\op{O}_{p}(S) \cap \operatorname{Z}(P)$ such that $Z_{1} \leqslant \operatorname{Z}(P) \cap P_0 \cap \op{O}_{p}(S)$. But then $Z_{1} \leqslant P_0 = \operatorname{C}_{P}(x)$, so $x\in\operatorname{C}_{N}(Z_{1})=1$, a contradiction. Therefore, we may assume that $\operatorname{Z}(P) \cap P_0 \cap \op{O}_{p}(S) =1$. On the other hand, we have $\op{O}_{p}(S) = \op{O}_{p}(S)\cap ZP_0 = Z(P_0\cap \op{O}_{p}(S)).$ If $P_0\cap \op{O}_{p}(S)$ is a non-trivial subgroup of $P$, since it is normal in $P$, we have a contradiction with $\operatorname{Z}(P) \cap P_0 \cap \op{O}_{p}(S) = 1$. Consequently we get the final contradiction $Z = \op{O}_{p}(S)$. The first assertion  is then established.

(2) Let  $K/N\op{O}_{p}(G) =  \op{O}_{p}(G/N\op{O}_{p}(G))$,  which has order $p$, and let $P$ be a Sylow $p$-subgroup of $K$. Then $K = PN$. Moreover, $[K, N]$ is normal in $G$ and $[K, N] = [P, N] \leqslant N$ so, by the minimality of $N$, we have either $[P, N] = 1$ or $[P, N] = N$. The first case leads to $K = P\times N$, and then $P\leqslant \op{O}_{p}(G)$, a contradiction. Thus we have $[P, N] = N$, and by coprime action it follows $\ce{N}{P}=1$.
\end{proof}

\medskip  
  
Note that every dihedral group of order $2q$ (for $q$ an odd prime) verifies the hypotheses of the above proposition (take $p=2$).

Theorem \ref{teoA} (3) is indeed an immediate consequence of the next more general result.

\begin{theorem}
\label{teoElementary}
Let $G = AB$ be a soluble group, which is the mutually permutable product of the subgroups $A$ and $B$. Assume that $G$ is $p$-nilpotent for a prime $p$. If $p^{2}$ does not divide $\abs{x^{G}}$ for any $p$-regular element $x\in A\cup B$ of prime power order, then $G/\op{O}_{p}(G)$ has elementary abelian Sylow $p$-subgroups.
\end{theorem}

\begin{proof}
Suppose that the result is false and let $G$ be a minimal counterexample. We may assume by the minimality of $G$ that $\op{O}_{p}(G)=1$, and therefore $\fra{G}=1$ too. By Theorem \ref{prodnotrivial}, we can assume that there exists a minimal normal subgroup $N$ of $G$ such that $N\leqslant A$. Moreover, $N$ is $q$-elementary abelian, for some prime $q\neq p$. Furthermore, since $N\cap \fra{G}=1$, by Gasch\"utz's lemma we may write $G = SN$, with $S\cap N = 1$. Let $P$ be a Sylow $p$-subgroup of $S$ (so it is a Sylow $p$-subgroup of $G$). Let $T=\op{O}_p(S)$. By the minimality of $G$ we have $T \cong \op{O}_{p}(G/N) \neq 1$, and by Proposition \ref{proptec} (1)  it holds $\abs{T}=p$. We may choose $1 \neq x \in N$ such that $ P_{0}=\op{C}_{P}(x)$ is a Sylow $p$-subgroup of  $\ce{G}{x}$. Since $\ce{N}{T}=1$ by Proposition \ref{proptec} (2), it holds that $P_{0} \neq P$, $\abs{P:P_0}=p$, and $P_{0} \cap T=1$. Hence $P = P_{0} \times T$. Finally, since $ P_0 \cong P/T \cong (PN/N)/\op{O}_{p}(G/N),$ which is elementary abelian by the minimality of $G$, it follows that $P$ so is, and this leads to the final contradiction. 
\end{proof}

\medskip
  
\begin{proof}[Proof of Theorem \ref{teoA}]
Note that the quotients of $G$ satisfy the hypotheses. Moreover, if $N$ is a normal subgroup of $G$ such that $N= (N \cap A)(N \cap B)$, then $N$ also inherits the hypotheses. (Observe that this occurs, for instance, if either $N \leqslant A$ or $N \leqslant B$.)

(1) We first see that $G$ is soluble by induction over $\abs{G}$. Since every group of odd order is soluble, we may affirm that $p=2$ because $\op{gcd}(p-1, \abs{G})=1$. By Theorem \ref{prodnotrivial}, we can assume that there exists a normal subgroup $M$ of $G$ such that $1\neq M\leqslant A$. If $M < G$, then $M$ is soluble by minimality. Analogously $G/M$ is also soluble, and then so is $G$. If $M=G$, we apply Theorem \ref{lemaqwa}.

(2) Suppose that the result is false and let $G$ be a counterexample of minimal order. Since the quotients of $G$ inherit the hypotheses, the class of $p$-nilpotent groups is a saturated formation, and $G$ is soluble, we may assume that $G$ possesses a unique minimal normal subgroup $N$, with $N = \ce{G}{N} = \fit{G}$. If  $N$ is a $p'$-group, since $G/N$ is $p$-nilpotent by the minimality of $G$, it follows that $G$ is $p$-nilpotent, which is a contradiction. Thus, we may assume that $N=\op{O}_p(G)$. By Theorem \ref{prodnotrivial}, we can assume without loss of generality that $N\leqslant A$, and that there exists a minimal normal subgroup $E/N$ of $G/N$ such that either $E/N \leq A/N$ or $E/N \leqslant BN/N$. In the first case, we have $E\leqslant A$. In the second case, it follows  $E = E\cap BN = N(E\cap B) \leqslant (E\cap A)(E\cap B) \leqslant E.$  Therefore, we have $E = (E\cap A)(E\cap B)$, where the factors are mutually permutable by Lemma \ref{mutperm} (b). In both cases, $E$ is normal in $G$ and $E$ verifies the hypotheses. Hence, if $E< G$, then $E$ is $p$-nilpotent by the minimality of $G$. Since $N=\op{O}_{p}(G)$, we get that $E/N$ is $q$-elementary abelian for some prime $q\neq p$, so it follows that $E=QN$, with $Q$ the normal Sylow $q$-subgroup of $E$. Hence, $Q$ is normal in $G$ which implies that $E=N$,  a contradiction.  

Therefore, we can assume that $E=G$. So we have $G = E = NQ$, where $Q$ is an abelian Sylow $q$-subgroup of $G$. By Lemma \ref{1.3.2}, we may assume that $Q = (Q\cap A)(Q\cap B)$, with either $Q\cap A \neq 1$ or $Q\cap B \neq 1$. Suppose first that $Q\cap B \neq 1$, and take $1\neq x\in Q\cap B$. Let $E_{1} = \langle x\rangle N $, which is normal in $QN = G$. Hence, we have $E_{1} = \langle x \rangle N \leqslant (E_{1}\cap B)(E_{1}\cap A) \leqslant E_{1}.$  If $Q\cap B=1$, then $G=A$ and we can choose $1 \neq x \in Q$, so that  $E_{1} = \langle x\rangle N$ is normal in $QN = G$. Thus, in both cases, we have that $E_{1}$ inherits the hypotheses and, if $E_{1} < G$, it follows that it is $p$-nilpotent. Therefore $\langle x \rangle$ is a normal Sylow $q$-subgroup of $E_1$, which is again a contradiction.
Consequently, we may assume that $G=E_{1}=\langle x \rangle N$, for some  $q$-element $x$. 

Note that $\operatorname{C}_{N}(x)$ is normal in $G = \langle x \rangle N$, since $N$ is abelian. By the minimality of $N$, it follows that either $\operatorname{C}_{N}(x) = 1$ or $\operatorname{C}_{N}(x) = N$. The second case leads to $x\in\operatorname{C}_{G}(N) = N$, a contradiction. Hence, it follows that $\operatorname{C}_{G}(x) = \operatorname{C}_{G}(x)\cap N\langle x\rangle  = \langle x \rangle\operatorname{C}_{N}(x) = \langle x\rangle.$ Then $\abs{x^{G}} = \abs{G:\operatorname{C}_{G}(x)} = \abs{N\langle x\rangle:\langle x \rangle} = \abs{N},$ and so $\abs{N}=p$,  by the hypotheses. Now, we get that $\langle x\rangle \cong G/N = \operatorname{N}_{G}(N)/\operatorname{C}_{G}(N)$ is isomorphic to a subgroup of $\operatorname{Aut}(N) \cong C_{p-1}$, the cyclic group of order $p-1$. Hence, $\abs{\langle x \rangle}$ divides both $p-1$ and $\abs{G}$, which contradicts the fact that $\operatorname{gcd}(p-1, \abs{G})=1$. This finishes the proof of the $p$-nilpotency of $G$.

(3) It follows from Theorem \ref{teoElementary}.
\end{proof}

\medskip

In the particular case when $G=A=B$ we recover:

\begin{corollary}\emph{(\cite[Theorem A]{QW})}
\label{CorTeoAQW}
Let $G$ be a  group. For a fixed prime $p$ with $\operatorname{gcd}(p - 1, \abs{G}) = 1$, if $p^{2}$ does not divide $\abs{x^{G}}$ for any $p$-regular element $x\in G$ of prime power order, then $G$ is soluble, $p$-nilpotent and $G/\op{O}_{p}(G)$ has elementary abelian Sylow $p$-subgroups.
\end{corollary}

Note that if $G$ is the direct product of two symmetric groups of degree 3, then $G$ verifies the hypotheses of Theorem \ref{teoA} for $p=2$, but not those of Corollary \ref{CorTeoAQW}. Moreover, the asumption that $\op{gcd}(p-1, \abs{G})=1$ is necessary, which can be seen by considering $G = A_5$, the alternating group of degree 5, and the prime $p=3$.

We include here a theorem due to Cossey and Wang \cite{CW}, which was the initial motivation for our results, to notify a gap that we have found in one of the statements. 

\begin{theorem}\emph{(\cite[Theorem 1]{CW})}
\label{teo1CW}
Let $G$ be a finite group, and $p$ be a prime divisor of $\abs{G}$ such that if $q$ is any prime divisor of $\abs{G}$, then $q$ does not divide $p-1$. Suppose that no conjugacy class size of $G$ is divisible by $p^2$. Then $G$ is a soluble $p$-nilpotent group, and $G/\op{O}_{p}(G)$ has a Sylow $p$-subgroup of order at most $p$. Further, if $P$
is a Sylow $p$-subgroup of G, then $P'$ has order at most $p$, and if $P \neq \op{O}_{p}(G)$, then $\op{O}_{p}(G)$ is abelian.
\end{theorem}

\begin{remark}
\label{remark1}
(a) The statement ``$G/\op{O}_{p}(G)$ has a Sylow $p$-subgroup of order at most $p$'' in the above theorem (and so the corresponding one in \cite[Theorem 6]{LWW}) is not true. 
	
To see this, consider the semidirect product $G = [C_5 \times C_5](\op{Sym}(3) \times C_2)$ (where $\op{Sym}(3)$ is a symmetric group of degree 3), which is the group of the Small groups library of \verb+GAP+ with identification number $300\#25$, and the prime $p=2$. Then $G$ verifies the hypotheses of Theorem \ref{teo1CW} but $\op{O}_{2}(G) =1$ and $\abs{G}_{2}=4$. We reveal that this example has been communicated to us by John Cossey.

(b) The same example shows that the hypotheses in Theorem \ref{teoA} for the conjugacy class sizes of the elements $x\in A\cup B$ are not necessarily inherited by the factors, unless they are (sub)normal in $G$. The above group $G$ can be factorised as the mutually permutable product of $A = D_{10} \times D_{10}$ and $B= [C_5 \times C_5]C_3$ (we checked this using \verb+GAP+). It is clear that $G=AB$ satisfies the hypotheses of Theorem \ref{teoA} for $p=2$, but there are elements $x\in A$ with $\abs{x^A}$ divisible by $4$.
\end{remark}

\begin{remark}
A natural question is how to extend the last assertion of Theorem \ref{teo1CW} for (mutually permutable) products. Concerning this, we show the following example:

Let $A=D_8$ be a dihedral group of order 8 and $B=[C_5]C_4 = \langle a, b \: \mid \: a^5=b^4=1, \: a^{b}=a^4\rangle$, and consider the prime $p=2$. Then $G = A\times B$ is a mutually permutable product of $A$ and $B$, and $G$ is $2$-nilpotent. Moreover, $4$ does not divide any conjugacy class size of elements in $A\cup B$. However, $\op{O}_{2}(G) = (\op{O}_{2}(G)\cap A)(\op{O}_{2}(G) \cap B) = D_8 \times C_2$ is not abelian.
\end{remark}

Regarding the claim ``$P'$ has order at most $p$'' in Theorem \ref{teo1CW}, we get the next extension for factorised groups, as an immediate consequence of Theorem \ref{knoche}:

\begin{corollary}
\label{corP'p-nilp}
Let $G = AB$ be the product of the subgroups $A$ and $B$. Assume that $G$ is $p$-nilpotent, and that for all $p$-elements in the factors, $p^2$ does not divide $\abs{x^{G}}$. If $P$ is a Sylow $p$-subgroup of $G$, then $P' \leqslant \fra{P} \leqslant \ze{P}$, with $P'$ elementary abelian of order at most $p^2$.
\end{corollary}

In particular, from this fact and Theorem \ref{teoA}, we get \cite[Theorem 7]{LWW} as a corollary, taking $G=A=B$.

Finally, we prove Theorem \ref{teoB}, which is motivated by \cite[Theorem 1.3]{BCL}.

\medskip

\begin{proof}[Proof of Theorem \ref{teoB}]
Suppose that the result is false and let $G$ be a counterexample of minimal order. Note that $G$ cannot be simple. Since the class of $p$-supersoluble groups is a saturated formation, we may assume that there exists a unique minimal normal subgroup $N$ of $G$, and that $\fra{G} = 1$. By the minimality of $G$, we get that $G/N$ is $p$-supersoluble. Since $G$ is $p$-soluble, it follows that $N$ is either a $p$-group or a $p'$-group. In the second case, since $G/N$ verifies the thesis by minimality, we get a contradiction. Consequently, we may assume that $N$ is $p$-elementary abelian and we must show that $\abs{N}=p$. As $\fra{G}=1$ and $G$ is $p$-soluble with $\op{O}_{p'}(G)=1$, by \cite[A - 10.6]{DH} it follows that $\fit{G}=\op{Soc}(G) = N = \ce{G}{N}$, and also $N=\op{O}_{p}(G)$. Applying Theorem \ref{prodnotrivial}, we may assume that there exists a minimal normal subgroup $Z/N$ of $G/N$ such that $Z/N \leqslant AN/N$, so $Z=Z\cap AN=N(Z\cap A).$ Since $G/N$ is $p$-soluble, it follows that $Z/N$ is either a $p$-group or a $p'$-group. The first case leads to $Z/N\leqslant\op{O}_{p}(G/N)=\op{O}_{p}(G/\op{O}_{p}(G))=1$, a contradiction. Hence, we may assume that $Z/N$ is a $p'$-group. 

Let $Q$ be a Sylow $q$-subgroup of $Z\cap A$, where $q\neq p$ is a prime (so $Q$ is a Sylow $q$-subgroup of $Z$). Therefore $Q \cong QN/N$ is a Sylow $q$-subgroup of $Z/N$, which acts faithfully on $N$. If $1\neq a \in Q \leqslant A$, then $N = [N, a] \times \operatorname{C}_{N}(a)$. By the hypotheses, since $p^2$ does not divide $\abs{a^{G}} = \abs{G:\operatorname{C}_{G}(a)}$, then neither divides $\abs{N:\operatorname{C}_{N}(a)} = \abs{[N, a]}$, so either $\abs{[N, a]} = 1$ or $\abs{[N, a]} = p$. The first case leads to $a\in\operatorname{C}_{G}(N) = N$, a contradiction. Thus, Lemma \ref{lemabcl} yields $QN/N$ is cyclic. Since this is valid for all primes $q\neq p$, we get by \cite[5.15]{ISA} that $Z/N$ is soluble. By the minimality of $Z/N$, it follows that $(Z/N)' =1$ and $Z/N$ is abelian with cyclic Sylow subgroups. Consequently $Z/N = \langle xN \rangle$, where $x\notin N$ and the order of $xN$ is $q$, for some prime $q\neq p$. 

We may assume that $x\in Z \cap A$ and $Z = N\langle x\rangle$. Hence $\operatorname{C}_{N}(x)=\operatorname{C}_{N}(Z)$. By the minimality of $N$, we have either $\operatorname{C}_{N}(Z) = N$ or $\operatorname{C}_{N}(Z)=1$. The first case leads to $x\in Z\leqslant \operatorname{C}_{G}(N)=N$, a contradiction. Therefore, since $N=[N, x]\times \operatorname{C}_{N}(x)$, it follows $\abs{N} = \abs{[N, x]} =p$, and this final contradiction establishes the theorem.
\end{proof}

\begin{example}
Let $G$ be the symmetric group of degree $4$. Then $G =AB$ is a mutually permutable product, where $A$ denotes the alternating group of degree 4 and $B$ is a Sylow 2-subgroup of $G$, which verifies the hypotheses of  Theorem \ref{teoB}, for $p=3$.
\end{example}


\section{Square-free class sizes}
\label{sec2}

We begin this section with the proof of Theorem \ref{teoD}.

\medskip

\begin{proof}[Proof of Theorem \ref{teoD}]
(1) Suppose that the result is false and let $G$ be a counterexample of least possible order. Since $G$ supersoluble,  $G/\fit{G}$ is abelian, and so $G' \leqslant \fit{G}$. Moreover, the quotients of $G$ inherit the hypotheses and the class of metabelian groups is a formation, so we may assume that there exists a unique minimal normal subgroup $N$ of $G$ with $\abs{N}=p$, for some prime divisor $p$ of $\abs{G}$. Hence, $\fit{G} = \op{O}_{p}(G) \leqslant P$, a Sylow $p$-subgroup of $G$. Since $G/\fit{G}$ is abelian, $P = \op{O}_{p}(G) = \fit{G}$. Hence $P = (P\cap A)(P\cap B)$, where $P\cap A$ and $P\cap B$ are Sylow $p$-subgroups of $A$ and $B$ respectively, by Lemma \ref{1.3.2}. Applying Theorem \ref{knoche}, we have $P' \leqslant \fra{P} \leqslant \ze{P}$, and $P'$ is elementary abelian of order at most $p^2$. Note that $P' \neq 1$, because $G' \leq P$.

By Lemma \ref{1.3.2}, we may consider $H$ a Hall $p'$-subgroup of $G$, such that $H = (H\cap A)(H\cap B)$, where $H\cap A$ and $H\cap B$ are Hall $p'$-subgroups of $A$ and $B$ respectively. Moreover, $H \cong G/\fit{G}$ is abelian. Let $x\in H\cap A$ be a prime power order element. Since $H\leqslant \ce{G}{x} \leqslant G$, it follows by the hypotheses that $\abs{x^{G}} \leq p$, and so $\fra{G}\leqslant \ce{G}{x}$. Thus $\fra{G}\leqslant \ce{G}{H\cap A}$, and analogously for $H\cap B$. Consequently we get $P' \leqslant \fra{P} \leqslant \fra{G} \leqslant \ce{G}{H}$. Since $P'\leqslant\fra{P}\leqslant\ze{P}$, it follows $P' \leqslant \fra{P} \leqslant\ze{G}$. In particular, $\fra{G} \neq 1 \neq \ze{G}$.

If $A, B\leqslant P$, then $G=P$ and $G' = P' \leqslant \ze{G}$, a contradiction. Hence we have either $H\cap A \neq 1$ or $H \cap B\neq 1$. Assume $H\cap A\neq 1$. Let $Q_A$ be a Sylow $q$-subgroup of $H\cap A$, for some prime $q \neq p$. Note that $Q_A \not \leqslant \ze{G}$, because $\ze{G}$ is a $p$-group. Let $\overline{Q_A}=Q_A\ze{G}/\ze{G}$, which acts on $\overline{P}= P/\ze{G}$, which is elementary abelian because $\fra{P}\leqslant\ze{G}\leqslant\fit{G}=P$. Suppose $\overline{w}=w\ze{G} \in \ce{\overline{Q_A}}{\overline{P}}$. Then $[\overline{w}, \overline{y}]=1$ for all $\overline{y}=y\ze{G} \in \overline{P}$, so $[w, y] \in \ze{G}\leqslant P$. Let  $k=o(w)$ denote the order of $w$. Thus $[w, y]^k = [w^k, y] = 1$. It follows $[w, y]=1$ for all $y\in P$, so $w\in \ce{Q_A}{P}=\ce{Q_A}{\fit{G}} =1$. Then $\ce{\overline{Q_A}}{\overline{P}}=1$ and the action is faithful. Let $1\neq \overline{\alpha_q}=\alpha_q\ze{G}\in\overline{Q_A}$. Therefore $\overline{P} = [\overline{P}, \overline{\alpha_q}] \times \ce{\overline{P}}{\overline{\alpha_q}}$, with $[\overline{P}, \overline{\alpha_q}] \neq 1$. Moreover, $\abs{[\overline{P}, \overline{\alpha_q}]} = \abs{\overline{P}:\ce{\overline{P}}{\overline{\alpha_q}}} = \abs{\overline{P}:\overline{\ce{P}{\alpha_q}}} = \abs{P:\ce{P}{\alpha_q}},$ which divides $\abs{\alpha_q^G}$, because $P$ is normal in $G$. Since $H \leqslant \ce{G}{\alpha_q}$ and $\alpha_q$ is a non-central prime power order element in $A$, it follows $\abs{\alpha_q^G}=p$, and so $\abs{[\overline{P}, \overline{\alpha_q}]} =p$. Applying Lemma \ref{lemabcl}, we get that $\overline{Q_A}\cong Q_A$ is cyclic. Since this is valid for each prime divisor $q$ of $\abs{H\cap A}$, we deduce that $H\cap A$ has cyclic Sylow subgroups, but it is abelian, so $H \cap A$ is cyclic. Analogously, if $H\cap B\neq 1$, then it is cyclic.

Let $H \cap A=\langle \alpha\rangle$. Assume first that $1\neq \alpha$ is a $q$-element, for some prime $q$, and that $G = P\langle\alpha\rangle$. By the above argument, $\langle\alpha\rangle\ze{G}/\ze{G}$ acts faithfully on $\overline{P}$, and $\abs{[\overline{P}, \overline{\alpha}]}=\abs{P:\ce{P}{\alpha}} = p$. Let $y \in P\setminus\ce{P}{\alpha}$. If $[\overline{y}, \overline{\alpha}] = 1$, then $[y, \alpha]\in\ze{G}\leqslant P$. Hence $[y, \alpha]^{o(\alpha)} = [y, \alpha^{o(\alpha)}] = 1$, so $[y, \alpha]=1$ and $y\in\ce{P}{\alpha}$, a contradiction. Then $[\overline{y}, \overline{\alpha}]\neq 1$, and since $\abs{[\overline{P}, \overline{\alpha}]} = p$, it follows $[\overline{P}, \overline{\alpha}] = \langle [\overline{y}, \overline{\alpha}]\rangle$. Therefore we have $\overline{[P, \alpha]} = [\overline{P}, \overline{\alpha}] = \langle[\overline{y}, \overline{\alpha}]\rangle = \overline{\langle[y, \alpha]\rangle},$ so $[P, \alpha] \leqslant [P, \alpha]\ze{G} = \langle[y, \alpha]\rangle\ze{G}$, and then $G' = P'[P, \alpha] \leqslant \langle[y, \alpha]\rangle\ze{G}$ which is abelian, a contradiction.

Hence, we may assume that, for every prime $q$, if $\alpha_q$ is the $q$-part of $\alpha$, then $P\langle\alpha_q\rangle < G$.  Note that $P\langle\alpha_q\rangle$ is normal in $G$, and  $P\langle\alpha_q\rangle = (P\langle\alpha_q\rangle \cap A)(P\langle\alpha_q\rangle\cap B)$. Therefore, by the minimality of $G$, it follows that $(P\langle\alpha_q\rangle)'$ is abelian. Notice that $(P\langle\alpha_q\rangle)' = P'[P, \alpha_q]$, since $P$ is normal in $G$.  Let  $K= P'[P, \alpha_q]$, which is an abelian normal subgroup of $G$,  and let $t\in[P,\alpha_q]\leqslant K$. Then Lemma \ref{isaac_lem} leads to $\abs{\ce{G}{\alpha_q}}<\abs{\ce{G}{[t, \alpha_q]}}$. If $p$ divides $\abs{{[t, \alpha_q]}^G}$ we get a contradiction, because $\abs{{[t, \alpha_q]}^G}<\abs{\alpha_q^G}=p$. Hence, $P\leqslant \ce{G}{[t, \alpha_q]}$ and $[t, \alpha_q]\in\ze{P}$, for each $t\in[P, \alpha_q]$. By coprime action, $P=[P, \alpha_q]\ce{P}{\alpha_q}$. Thus $[P, \alpha_q] = [[P, \alpha_q]\ce{P}{\alpha_q}, \alpha_q] = [P, \alpha_q, \alpha_q]$. If $k$ is a generator of $[P, \alpha_q, \alpha_q]$, then $k = [t, \alpha_q]\in\ze{P}$ with $t\in[P, \alpha_q]$, so $[P, \alpha_q] =[P , \alpha_q, \alpha_q]\leqslant\ze{P}$. Since this is valid for each prime divisor $q$ of the order of $H \cap A=\langle \alpha\rangle$, we get:  $$[P, H\cap A] = [P, \langle\alpha_{q_1}\rangle \times \cdots \times \langle\alpha_{q_t}\rangle] = [P, \alpha_{q_1}]\cdots[P, \alpha_{q_t}] \leqslant \ze{P}.$$ Analogously, if $H\cap B \neq 1$, then $[P, H\cap B] \leq \ze{P}$. Since $G' = P'[P, H]= P'[P, H\cap A][P, H\cap B]$, we get $G' \leq \ze{P}$. This final contradiction establishes statement (1).

(2) Suppose that the second assertion is not true and let $G$ be a counterexample of minimal order. We point out that the hypotheses are inherited by every quotient group of $G$ and, by (1), $G'$ is abelian. There exists a prime divisor $p$ of $\abs{G'}$ such that $G'$ does not have any elementary abelian Sylow $p$-subgroup. By the minimality of $G$, we may consider that $\op{O}_{p'}(G)=1$. Moreover, since $G$ is supersoluble, then $G/\fit{G}$ is abelian, and $\fit{G} = \op{O}_p(G)= P$ is a normal Sylow $p$-subgroup of $G$ such that $G' \leqslant P$. Using Lemma \ref{1.3.2} and Theorem \ref{knoche}, we obtain respectively that $P=(P \cap A)(P \cap B)$, and that $P'$ is elementary abelian with $P' \leqslant \fra{P} \leqslant \ze{P}$.

Let $G^\mathfrak{N}$ be the nilpotent residual of $G$. Note that $G^\mathfrak{N}\neq 1$; in other case, $G$ is a $p$-group and then $G'=P'$, a contradiction. Since $G^\mathfrak{N} \leqslant G'$, it follows that $G^\mathfrak{N}$ is abelian. By using \cite[III - 4.6, IV - 5.18]{DH}, we have that $G^\mathfrak{N}$ is complemented in $G$, and its complements are precisely the Carter subgroups of $G$. Accordingly, $G=G^\mathfrak{N}H$ with $H=\op{N}_{G}(H)$ a nilpotent subgroup of $G$ and $G^\mathfrak{N} \cap H=1$. These facts yield $G'=G^\mathfrak{N} \times (H \cap G')$, and $\op{C}_{G^\mathfrak{N}}(H)=1$. On the other hand, the minimality of $G$ implies that $G'/G^\mathfrak{N} \cong (H \cap G')$ is elementary abelian, and thus $G^\mathfrak{N}$ is not so.  If $\ce{H}{G^\mathfrak{N}} \neq 1$, since $\ce{H}{G^\mathfrak{N}}$ is normal in $G$, by the minimality of $G$ we have $(G/\ce{H}{G^\mathfrak{N}})'$ is an elementary abelian group, but 
$$
(G/\ce{H}{G^\mathfrak{N}})' = G'\ce{H}{G^\mathfrak{N}}/\ce{H}{G^\mathfrak{N}} =  G^\mathfrak{N}\ce{H}{G^\mathfrak{N}}/\ce{H}{G^\mathfrak{N}} \cong G^\mathfrak{N},
$$
which is a contradiction. Hence, $\ce{H}{G^\mathfrak{N}}=1$.  In particular, we deduce that  $\ze{G}=1$.  

By Lemma \ref{1.3.2}, there exists a Hall $p'$-subgroup $H_0$ of $G$ such that $H_0=(H_0 \cap A)(H_0 \cap B)$. Let $x \in H_0 \cap A$ be a non-trivial element of prime power order. Since $H_{0}\leqslant \ce{G}{x}$, it follows that $\abs{x^{G}} \leq p$, and so $\fra{G}\leqslant \ce{G}{x}$. Thus $\fra{G}\leqslant \ce{G}{H_{0}\cap A}$, and analogously for $H_{0}\cap B$. Consequently, we get $P' \leqslant \fra{P} \leqslant \fra{G} \leqslant \ce{G}{H_{0}}$. Since $P'\leqslant\fra{P}\leqslant\ze{P}$ and $G=PH_{0}$, it follows $P' \leqslant \fra{P} \leqslant \ze{G}=1$, which implies that $P$ is elementary abelian, the final contradiction. 

(3) Assume that the result is false and take $G$ a counterexample of minimal order. Consider  a prime $p$ such that $|\fit{G}'|_{p} > p^{2}$. By minimality, we can  affirm that  $\op{O}_{p'}(\fit{G})=1$.  Since $G$ is supersoluble, we get that  $\fit{G} = \op{O}_p(G)= P$ is a normal Sylow $p$-subgroup of $G$. Then, we apply both Lemma \ref{1.3.2} and  Theorem \ref{knoche} to get the final contradiction. 
\end{proof}

\begin{example}
Let $G =A \times B$ be the direct product of two symmetric groups of degree 3. Then $G$ is supersoluble, and every element contained in each factor (not only those of prime power order) has square-free conjugacy class size, but neither the derived subgroup $G'$ nor $G/\fit{G}$ are cyclic, in contrast to \cite[Theorem 2]{CW}. 
\end{example}

\begin{example}
In view of \cite[Theorem 2]{CW}, it is natural to wonder if we can affirm in the above result that the Sylow $p$-subgroups of $G'$ have order at most $p^2$. This fact is not further true, as we show:

Let $G=A \times B$, where $A = D_{14}$ is a dihedral group of order 14, and $B = D_{14} \times [C_7]C_3$ is the direct product of such a dihedral group and a semidirect product of a cyclic group of order 7 and a cyclic group of order 3 ($B$ has identification number $294\#9$ in the Small groups library of \verb+GAP+). Then $G$ is supersoluble, and satisfies that all prime power order elements contained in each factor have square-free conjugacy class size, but $G'$ has order $7^3$. 
\end{example}

Now we proceed with the proof of Theorem \ref{teoC}.
  
\medskip

\begin{proof}[Proof of Theorem \ref{teoC}]
Considering the smallest prime divisor of $\abs{G}$ and Theorem \ref{teoA}, we conclude that $G$ is soluble. Hence, it is $p$-soluble for each prime divisor $p$ of $\abs{G}$. Applying Theorem \ref{teoB}, we get that $G$ is $p$-supersoluble for each prime that divides $\abs{G}$, so it is supersoluble.

Now we prove the second assertion by induction on $\abs{G}$. Let $p$ be an arbitrary prime, and $P$ be a Sylow $p$-subgroup of $G$. We want to show that $P\op{F}(G)/\op{F}(G)\cong P/\op{O}_{p}(G)$ is elementary abelian. Since $G$ is supersoluble, we have that $G/\op{F}(G)$ is abelian. Moreover, we may assume by induction that $\op{O}_{p}(G)=1$. Therefore, we have that $\op{F}(G) \leqslant H \leqslant G,$ where $H$ is a Hall $p'$-subgroup of $G$. Consequently, $H$ is normal in $G$ and $G$ is $p$-nilpotent. Finally, by Theorem \ref{teoElementary} the result is established.
\end{proof}

\medskip

When considering in the above theorem all $p$-regular elements in the factors, we get as a corollary:

\begin{corollary}\emph{(\cite[Corollary 1.5]{BCL})}
\label{cor1.5BCL}
Let the group $G = AB$ be the mutually permutable product of the subgroups $A$ and $B$. Suppose that for every prime $p$ and every $p$-regular element $x \in A \cup B$, $\abs{x^G}$ is not divisible by $p^2$. Then $G$ is supersoluble.
\end{corollary}

\begin{example}
Consider $G=A\times B$, where $A= \op{Sym}(3)$ is a symmetric group of degree 3, and $B = \op{Sym}(3) \times D_{10}$ is the direct product of such a symmetric group and a dihedral group of order 10. Then $G$ satisfies the hypotheses of  Theorem \ref{teoC}. However there exists some $2$-regular element in $B$, not of prime power order, such that  $4$ divides its conjugacy class size, so Corollary \ref{cor1.5BCL} cannot be applied.
\end{example}

In the particular case when $A$ and $B$ are normal in $G$, we obtain \cite[Proposition 9]{LWW}.

\begin{corollary}
Let $A$ and $B$ be normal subgroups of $G$ such that $G=AB$. Suppose that $\abs{x^G}$ is square-free for every element $x$ of prime power order of $A\cup B$. Then $G$ is supersoluble.
\end{corollary}

This development has its origins in the contributions of Chillag and Herzog \cite[Theorem 1]{CH}, and  Cossey and Wang \cite[Theorem 2]{CW}. Our next result Theorem \ref{teoE} and Theorem \ref{teoD} can be considered somehow extensions of the ones above for (mutually permutable) products. In fact, Theorem \ref{teoE} provides further information on the Sylow subgroups of $G/\fit{G}$.

\medskip

\begin{proof}[Proof of Theorem \ref{teoE}]
Suppose the result is not true and let $G$ be a counterexample of least order possible. Then if $P$ is a Sylow $p$-subgroup of $G$, we have $\abs{P/\op{O}_p(G)}\geq p^3$. We can assume by the minimality of $G$ that $\op{O}_{p}(G)=1=\fra{G} = \ze{G}$, so $\abs{P}\geq p^3$. By Lemma \ref{1.3.2} we can choose $P = (P\cap A)(P\cap B)$, with $P\cap A$ and $P\cap B$ Sylow $p$-subgroups of $A$ and $B$ respectively. By Theorem \ref{teoC}, we have that $G$ is supersoluble and $G/\op{F}(G)$ has elementary abelian Sylow subgroups. In particular, $P\op{F}(G)$ is normal in $G$. Hence Lemma \ref{mutperm} (b) asserts that $L = (P\op{F}(G)\cap A)(P\op{F}(G)\cap B)$ is normal in $G$, and it is a mutually permutable product. Moreover, $P = (P\cap A)(P\cap B)\leqslant (P\op{F}(G)\cap A)(P\op{F}(G)\cap B) = L \leqslant P\fit{G}.$ If we suppose $L < G$, by the minimality of $G$ it follows $\abs{P/\op{O}_{p}(L)} = \abs{P}\leq p^2$, a contradiction. Thus, we may assume $L=G=P\op{F}(G)$, and so $G$ is $p$-nilpotent. 

Let $N$ be a minimal normal subgroup of $G$. We can assume without loss of generality that it is contained in $A$ by Theorem \ref{prodnotrivial}. Note $\abs{N}=q\neq p$. By Proposition \ref{proptec} (1), it follows $\abs{\op{O}_{p}(G/N)}\leq p$, and by the minimality of $G$ we have $\abs{(PN/N)/\op{O}_{p}(G/N)}\leq p^2.$ Since $P\cong PN/N$, we may assume $\abs{P}=p^3$. As $\fra{G}=1$, [\cite{DH}, A - 10.6 Theorem] leads to $\fit{G}=\op{Soc}(G)$. If $N$ is the unique minimal normal subgroup of $G$, then $P \cong \op{N}_{P}(N)/\op{C}_{P}(N)$ which is isomorphic to a subgroup of $\op{Aut}(C_q) \cong C_{q-1},$ so $P$ is cyclic and elementary abelian, which implies that its order is $p$, a contradiction.

Now we denote by $T$ the product of all minimal normal subgroups of $G$ distinct of $N$, so $T \neq 1$ and $T\cap N =1$. It follows $\op{F}(G) = \op{Soc}(G) = N \times T$. We denote $Q_1= \op{O}_{p}(PN)$ and $Q_2= \op{O}_{p}(PT)$. Since $PN\cong G/T$ (and $PT\cong G/N$), by the minimality of $G$ we may affirm $Q_1 \neq 1 \neq Q_2$. On the other hand, since $[\op{O}_{p}(PN), N] \leqslant \op{O}_{p}(PN)\cap N = 1$, we have $Q_1\leqslant C_1= \op{C}_{P}(N)$ (analogously $Q_2 \leqslant C_2= \op{C}_{P}(T)$). In addition, it follows $C_1 \cap C_2 \leqslant \op{C}_{P}(\op{F}(G)) \leqslant \op{C}_{G}(\op{F}(G)) \leqslant \op{F}(G), $ so $C_1 \cap C_2 =1$. Let $P_0$ be a Sylow $p$-subgroup of $\op{C}_{G}(N)$ such that $P_0 \leqslant P$. Hence $P_0=C_1=\op{C}_{P}(N)$. In addition, since $N = \langle x\rangle$ where $x$ is a $q$-element contained in $A$, by the hypotheses it follows that $p^2$ does not divide $\abs{x^G}_{p}=\abs{G:\op{C}_{G}(x)}_{p} = \abs{P:C_1}$. Moreover, since $\abs{P} = \abs{P:C_{1}}\cdot\abs{C_{1}} = p^3$, we may assume $\abs{C_1}\geq p^2$, and since $1 \neq Q_2 \leqslant C_2$, we have $\abs{C_2}\geq p$. Accordingly $\abs{C_1 C_2} = \abs{C_1}\cdot\abs{C_2} \geq p^3$, and since $P$ is abelian, we have necessarily $P=C_1\times C_2$. This leads to $$ G = P\op{F}(G) = C_1 C_2 N T = (C_1 T) \times (C_2 N). $$ 

Suppose $T\cap A\neq 1$, and let $1\neq y \in T\cap A \leqslant C_1T\cap A$. Let $1\neq x\in N \leqslant A$. Then since $xy \in \fit{G}\cap A$ and $\fit{G}$ is abelian, we have that $xy$ is a $p$-regular element, so by the hypotheses $p^2$ does not divide $\abs{(xy)^G}$. As $G$ is a direct product, we have $\abs{(xy)^G} = \abs{x^G}\abs{y^G}$. In addition, $(N\times C_1T)\leqslant \ce{G}{x}\leqslant G$, and therefore $\abs{x^G}$ divides $\abs{G:(C_1T\times N)} = \abs{C_2}$ which is a $p$-number, so $\abs{x^G}=p$ (recall that $\ze{G}=1$), and analogously $\abs{y^G} = p$, a contradiction. We conclude $T\cap A = 1$. 

If $\op{F}(G) = K \times M$ with $M$ normal in $G$ and $K$ a minimal normal subgroup of $G$ contained in $B$, by similar arguments we can deduce $M \cap B=1$. This means, in particular, that neither $A$ nor $B$ can contain two distinct minimal normal subgroups of $G$.

On the other hand, since $\op{F}(G)$ is the unique $p'$-Hall subgroup of $G$, Lemma \ref{1.3.2} leads to $\op{F}(G) = (\op{F}(G) \cap A)(\op{F}(G)\cap B).$ Moreover, since $\op{F}(G) \cap A = N T \cap A = N(T\cap A) = N$, it follows $\op{F}(G) = N(\op{F}(G)\cap B)$. Note that $(\op{F}(G) \cap B) \cap N \leqslant N$ with $\abs{N}=q$ so we distinguish two cases: either $(\op{F}(G) \cap B) \cap N = N$ or $(\op{F}(G) \cap B) \cap N = 1$. In the first case $\op{F}(G) = \op{F}(G) \cap B\leqslant B$. Thus there exists another minimal normal subgroup contained in $B$ and distinct of $N$, a contradiction. Hence we conclude $(\op{F}(G) \cap B) \cap N = 1$ so $\op{F}(G) = N \times (\op{F}(G) \cap B)$. 

Now suppose that $\op{F}(G)$ is a $q$-group. Then, since $\op{F}(G) = \op{Soc}(G)$, it follows that $\op{F}(G)$ is $q$-elementary abelian. In addition, $P\cap A$ acts faithfully over $\op{F}(G)$. Let $1\neq x\in P\cap A \leqslant A$, then $\op{F}(G) = [\op{F}(G), x] \times \op{C}_{\op{F}(G)}(x)$, with $\op{C}_{\op{F}(G)}(x) < \op{F}(G)$ since $\op{Z}(G) = 1$ and $P$ is abelian. By the hypotheses, $q^2$ does not divide $\abs{x^{G}}$, and therefore it does not divide $\abs{\op{F}(G):\op{C}_{\op{F}(G)}(x)}$. Thus we may affirm $\abs{[\op{F}(G), x]} =q$. By Lemma \ref{lemabcl} we conclude that $P \cap A$ is cyclic, and analogously $P \cap B$ is cyclic too. So they are both cyclic and elementary abelian, that is, they both have order $p$. Thus $\abs{P} = \abs{(P\cap A)(P\cap B)}\leq p^2,$ a contradiction. 

Hence we may suppose that there exists a prime $r\neq q$ such that $r$ divides $\abs{\op{F}(G)}$. Let $1\neq R$ be a Sylow $r$-subgroup of $\op{F}(G) \cap B$ (so it is a Sylow $r$-subgroup of $\op{F}(G)$). Then $1\neq R=\op{O}_{r}(G) \leqslant B$, and since $\op{F}(G) = \op{Soc}(G)$, necessarily we have that $\op{O}_{r}(G)$ is the product of the minimal normal subgroups of $G$ with order $r$. Let $M \leqslant B$ be one of those minimal normal subgroups. Arguing exactly in the same way as with $N$, it follows $\op{F}(G) = M \times (A\cap \op{F}(G))$. But $A\cap \op{F}(G) = N$ so $\op{F}(G) = N \times M$ with both minimal normal subgroups of $G$, $N\leqslant A$ and $M\leqslant B$. Let $P_1$ be a Sylow $p$-subgroup of $\op{C}_{G}(N) = \op{C}_{G}(x)$ such that $P_1\leqslant P$. Then by the hypotheses we have $\abs{x^G}_{p}=\abs{P:P_1} \leq p$. Since $\abs{P}=p^3$, it follows $\abs{P_1}\geq p^2$. However, $P_1$ is normal in $PN$ so $P_1\leqslant \op{O}_{p}(PN) \cong \op{O}_{p}(G/M)$, and by Proposition \ref{proptec} (1) we have $\abs{\op{O}_{p}(G/M)} \leq p$. This final contradiction establishes the theorem.
\end{proof}

\begin{example}
\label{exampleG/FittingPrimePower}
Under the hypotheses of Theorem \ref{teoC} (even under those of Theorem \ref{teoD}), it is not possible to assure that $G/\op{F}(G)$ has Sylow $p$-subgroups of order at most $p^2$, as the following example shows:

Let $\lbrace p_1, p_2, \ldots, p_n\rbrace$ be a finite set of pairwise distinct odd primes, and let $G = D_{2p_1}\times D_{2p_2} \times \cdots \times D_{2p_n}$ be the direct product of dihedral groups of order $2p_i$, $1\leq i \leq n$. Then $G=A \times B$ is a mutually permutable product of $A=D_{2p_1}$ and $B=D_{2p_2} \times \cdots \times D_{2p_n}$, and each prime power order element contained in the direct factors has square-free conjugacy class size. However, $G/\fit{G}$ has order $2^n$.
\end{example}

\end{document}